\documentclass[11pt]{article}
\usepackage{amsmath,amssymb,amsthm}
\usepackage[a4paper, margin=2.5cm]{geometry}
\numberwithin{equation}{section}

\title{Low-regularity well-posedness for the ZK Equation on a half-strip}
\author{E Avelino, G Doronin}
\date{}

\newtheorem{theorem}{Theorem}
\newtheorem{proposition}{Proposition}
\newtheorem{lemma}{Lemma}

\begin{document}

\maketitle

\begin{abstract}
Studied here is the Zakharov--Kuznetsov equation with a linear transport
term posed on a half-strip with nonhomogeneous boundary condition.
Using Bourgain-type spaces adapted to the ZK dispersive structure,
anisotropic smoothing and boundary trace estimates, we establish its local well-posedness in $L^2.$
\end{abstract}

\section{Introduction}

The Zakharov--Kuznetsov (ZK) equation
\begin{equation}
\label{ZK} u_{t} + \partial_{x}\left(\Delta u + a u\right) +\frac12
\partial_x(u^2)=0
\end{equation}
arises as a multidimensional generalization of the Korteweg--de
Vries (KdV) equation in plasma physics and describes the propagation
of weakly nonlinear ion-acoustic waves in a magnetized plasma,
\cite{zk}. Equation \eqref{ZK} exhibits a strongly anisotropic
dispersive structure which combines one-dimensional KdV dispersion
in the longitudinal direction with transversal effects, \cite{LP}.

We study the initial-boundary value problem for \eqref{ZK}
posed on the half-strip $\{(x,y)\in \mathbb{R}^{2}: x>0,\ 0<y<B\}$ supplemented with
Dirichlet boundary conditions in the transversal variable, and a
nonhomogeneous boundary condition (wavemaker) at $\{x=0\}.$ More
precisely, we consider
\begin{equation}\label{IBVP}
\begin{cases}
u_{t} + u_{xxx} + u_{xyy} + a u_{x} +uu_{x}=0,\ t>0;\\
u(0,y,t)=g(t,y);\\ u(x,0,t)=u(x,B,t)=0;\\ u(x,y,0)=u_0(x,y).
\end{cases}
\end{equation}

As concerns the pure IVP posed on the whole $\mathbb{R}^{2}$, we cite \cite{F95, GH, FP}, and the best known result is probably \cite{kinoshita}. The local well-posedness in $H^{s}(\mathbb{R}^{2})$ has been proved there for $s>-1/4$. The scaling argument, however, indicates that even this negative regularity can be weaked, \cite{SL}. Therefore, the low regularity well-posedness for the ZK equation is strongly wellcomed.

Note that problems for dispersive equations posed on domains
with boundaries present substantial analytical difficulties due to
the interaction between dispersion and boundary effects, \cite{BF}. For the KdV
equation on the half-line, a systematic framework based on boundary
forcing operators was developed in \cite{BSZ, CK}, allowing to
convert the boundary condition into a forcing term supported at the
boundary, and to work in Bourgain-type spaces.

For the ZK equation, the presence of the mixed derivative $u_{xyy}$
leads to additional complications. In particular, the dispersive
relation generates a two-dimensional resonant geometry which apparently weakens
the smoothing effects compared to the one-dimensional KdV case,
\cite{GH, FP, LP, SL, MP, MR}. 
As a consequence, usual (2+1)D extensions of standard
Bourgain spaces $X^{s,b}$ are insufficient to control the
nonlinearity even at the $L^2$ level.

A systematic study of well-posedness for IBVPs on a half-strip in
the weighted Sobolev spaces has been developed in \cite{F}. The
regularity of solutions there seems not to be sharp however. Indeed, whereas the scale argument suggests $s>-1$, the IBVP \eqref{IBVP} with $u_{0}\in L^{\rho(x)}_{2}$ and $g\in H^{s/3,s}$, $s>3/2$, is shown in \cite{F} to be well-posed in weighted $L^{2}$-based Sobolev spaces $X^{\rho(x)}_{w}$. Morever, the specific weight $\rho(x)\leq C(\rho'(x))^{2}$ is essential to prove the uniqueness.

To put down the possible solutions regularity, we employ here a modified functional setting
which combines Bourgain-type spaces with energy estimates. We use
also corresponding Dirichlet configuration of boundary operators
which is physically quite reasonable and admits a good
discretization frame. More detailed, we work in the space
$$
Z^{b}_{T}:= X^{0,b}_{T} \cap L^{2}(0,T;H^{1}),
$$
which incorporates the classical Bourgain space $X^{s,b}$ and the spatial Kato-type smoothing effect available for the
linear problem. This additional structure compensates the derivative
present in the nonlinearity $\partial_x(u^2)$ and allows to
recover the necessary bilinear estimates.

The main result of this paper establishes a local well-posedness for
\eqref{IBVP} with initial data in $L^2$ and boundary data in
$H^{1/3}_t L^2_y.$ The proof relies on a construction of boundary
forcing operators, a spectral reduction in the
transversal variable, and Bourgain estimates adapted to the ZK
dispersion. Bilinear estimates are obtained in the modified space
$Z^b_T,$ and a contraction mapping is applied as a final step.

A key idea follows \cite{CK} and is that the boundary wavemaker can
be incorporated into a boundary forcing term in the corresponding
Cauchy problem on the whole $\mathbb{R}^{2}.$ It does not affect the
resonance structure of the equation; therefore, it does not
influence the regularity thresholds. The method can be viewed as a
natural extension of the framework from \cite{CK} to the ZK equation
in a bounded transversal geometry, and it can be used for posterior
negative regularity analysis, as well as for gZK consolidation.

\section{Notations and results}

\subsection{Specific notations}

Throughout the paper, $C$ and $c$ denote various positive constants whose exact values may change depending on the context. We use the notation $A \sim B$ to indicate that there exist positive constants $c$ and $C$ such that $c B \leq A \leq C B$.

We denote the spatial domain by $\Omega = \mathbb{R}^{+} \times (0,B)$. For the transverse variable $y \in (0,B)$, we consider $\{e_k\}_{k \ge 1}$ to be $e_{k}=\sqrt{2/B}\sin (\sqrt{\lambda_{k}}y)$ with $\lambda_{k}=(k\pi/B)^{2}$. Given a function $u(x,\cdot,t)\in L^{2}(0,B)$, its Fourier series expansion in the transverse variable will be described by its coefficients (or modes) $u_k(x,t) = \langle u(x,\cdot,t), e_k \rangle_{L^2_y}$ as $u=\sum_{k\geq1} u_{k}e_{k}$.

We use $\mathcal{F}$ or $\,\widehat{\cdot}\,$ to denote the space-time Fourier transform $\widehat{u}=\mathcal{F}[u]=\int e^{-i\langle\xi, s\rangle}u\,ds$. Specifically, we will denote by $\widehat{u}_k(\xi,\tau)$ the Fourier transform in the continuous variables $(x,t) \in \mathbb{R}^2$ of the transverse mode $u_k$, where $\xi$ and $\tau$ represent the spatial and temporal dual variables, respectively.

We denote by $S(t) := e^{-t\partial_x(\Delta + a)}$ the unitary group generated by the linear part of the ZK equation on the whole $\mathbb{R}^{2}$, whose 2D dispersion relation is given by $\omega(\xi,\eta) = \xi^3 + \xi\eta^2 - a\xi$. The spectral projection onto the $k$-th transverse mode induces the corresponding one-dimensional propagator $S_k(t) := e^{it\omega_k(\xi)}$, associated with the parameterized dispersion relation $\omega_k(\xi) = \xi^3 + (\lambda_k - a)\xi$. The fundamental solution of the linear operator will be written in terms of the two-dimensional Airy function, defined by
\begin{equation}\label{airyeq2d}
\hat{A}(\xi, \eta)=e^{i\xi^{3}}e^{i\xi\eta^{2}}.
\end{equation}

We choose $\theta \in C_0^\infty(\mathbb{R})$ to denote a smooth temporal cutoff function, with $0 \le \theta \le 1$, supported in a neighborhood of the origin. The Dirac delta distribution centered at the origin is denoted by $\delta_0$. For any real $x$, we use $\langle x \rangle := 1 + |x|$.

\subsection{Function spaces}

For $s \in \mathbb{R}$, $H^s$ refers to the classical $L^2$-based Sobolev space, whose norm is defined via the spatial Fourier transform
$$\|f\|_{H^s} := \left( \int \langle \xi \rangle^{2s} |\widehat{f}(\xi)|^2 d\xi \right)^{1/2}.$$

If $U$ is a Banach space equipped with the norm $\|\cdot\|_{U}$ and $T>0$, the notation $L^p_T U$ abbreviates the Bochner space $L^p(0,T; U)$ of strongly measurable functions $u: [0,T] \to U$, equipped with the norm:
$$\|u\|_{L^p_T U} := \left( \int_0^T \|u(t)\|_{U}^p dt \right)^{1/p}, \quad 1 \leq p < \infty,$$
with the usual modification for the $\operatorname{ess\, sup}$ norm when $p=\infty$. Similary,
$$\|u\|_{L^p_t L^q_{x,y}} := \left( \int_{\mathbb{R}} \left( \iint_{\Omega} |u(x,y,t)|^q dx dy \right)^{p/q} dt \right)^{1/p}.$$

Analogously, the fractional temporal space $H^{s}_{t}$ conjugated with the transverse space, say $L^{2}_{y}$, is denoted by $H^{s}_{t} L^{2}_{y}$, with its norm expressed via the Fourier transform restricted to the temporal variable
$$\|u\|_{H^s_t L^2_y} := \left( \int_{\mathbb{R}} \langle \tau \rangle^{2s} \|\mathcal{F}_t u(\cdot,\cdot,\tau)\|_{L^2_y}^2 d\tau \right)^{1/2}.$$

We follow the modified Bourgain spaces definition from \cite{CK, MP}. Let $\mathcal{S}'(\mathbb{R}^2 \times \mathbb{R})$ be the space of tempered distributions. For $u \in \mathcal{S}'(\mathbb{R}^2 \times \mathbb{R})$ whose transverse expansion is given by $u(x,y,t) = \sum_{k=1}^{\infty} u_k(x,t)e_k(y)$, we define the space $X^{s,b}$ as the closure of Schwartz functions under the norm
\begin{equation}
\begin{aligned}
\|u\|_{X^{s,b}}^2 := \sum_{k=1}^{\infty} \Bigg( &\underset{|\omega_k(\xi)| \ge 1}{\iint} \langle \tau-\omega_k(\xi)\rangle^{2b}\langle 3\xi^{2}+\lambda_{k}\rangle^{2s} |\widehat{u}_k(\xi,\tau)|^2 d\xi d\tau \\ &+ \underset{|\omega_k(\xi)| < 1}{\iint} \langle \tau \rangle^{2\alpha}\langle3\xi^{2}+\lambda_{k}\rangle^{2s} |\widehat{u}_k(\xi,\tau)|^2 d\xi d\tau \Bigg),\end{aligned}
\end{equation}
where $0<b<1/2$ and $1/2<\alpha <2/3$. The corresponding space for the non-homogeneous (or forcing) term, denoted by $Y^{s,b}$, is defined analogously by the norm
\begin{equation}
\begin{aligned}
    \|f\|_{Y^{s,b}}^2 := \sum_{k=1}^{\infty} \Bigg( &\underset{|\omega_k(\xi)| \ge 1}{\iint} \frac{\langle3\xi^{2}+\lambda_{k}\rangle^{2s}|\widehat{f}_k(\xi,\tau)|^2}{\langle \tau-\omega_k(\xi)\rangle^{2b}} d\xi d\tau \\ &+ \underset{|\omega_k(\xi)| < 1}{\iint} \frac{\langle3\xi^{2}+\lambda_{k}\rangle^{2s}|\widehat{f}_k(\xi,\tau)|^2}{\langle \tau \rangle^{2(1-\alpha)}} d\xi d\tau\\
    &+ \underset{|\omega_k(\xi)| \ge 1}{\int}\Bigg(\underset{|\omega_k(\xi)| \ge 1}{\int} \frac{\langle3\xi^{2}+\lambda_{k}\rangle^{s}|\widehat{f}_k(\xi,\tau)|}{\langle \tau-\omega_k(\xi)\rangle} d\tau\Bigg)^{2} d\xi\Bigg).
    \end{aligned}
\end{equation}
For a finite $T>0$, the restrictions of the above spaces to the temporal interval $[0,T]$ will be denoted with a subscript $T$, like $X^{s,b}_{T}$, $Y^{s,b}_{T}$ and $Z^{b}_{T}$, equipped with the norm
\begin{equation}\|u\|_{X^{s,b}_T} := \inf \big\{ \|v\|_{X^{s,b}} \,:\, v \in X^{s,b} \text{ with } v|_{[0,T]} = u \big\}.\end{equation}

\subsection{Main result}\label{sub2.3}

Finally, the solution space $Z^b_T$ for the local theory is structured as the intersection of the modified Bourgain space with the natural energy space of the equation. For $0< b < 1/2$, we define
\begin{equation}
Z^b_{T} := X^{0,b}_T \cap L^2(0,T; \dot{H}^1_{x,y}),\ \dot{H}^{1}_{x,y}=H^{1}(\mathbb{R}_{+})\times H^{1}_{0}(0,B),\end{equation}
equipped with the sum norm $\|u\|_{Z^b_T} = \|u\|_{X^{0,b}_T} + \|u\|_{L^2_T \dot{H}^1_{x,y}}$.

Our main result is

\begin{theorem}{Theorem}\label{Localwellposedness}
For any $u_{0}\in L^{2}(\Omega)$ and $g\in H^{1/3}_{t}(0,T; L^{2}_{y}(0,B))$, there exist $T_{0}>0$, and a unique solution $u(x,y,t)$ to \eqref{PVIC} such that
\begin{equation*}
u\in C([0,T_{0}];L^{2}(\Omega))\cap Z^{b}_{T_{0}}.
\end{equation*}
Furthermore, the data-to-solution map $(u_{0},g) \mapsto u$ is Lipschitz continuous.
\end{theorem}

\section{Reducing to the Cauchy problem}

Recall that we consider the initial boundary value problem (IBVP) for the Zakharov-Kuznetsov equation defined on the spatial half-strip $\Omega = (0, \infty) \times (0, B)$, given by:
\begin{equation}\label{PVIC}\begin{cases}
u_t + au_{x} + u_{xxx} + u_{xyy} + u u_x = 0, & (x,y) \in \Omega, \ t \in (0, T), \\
u(x,y,0) = u_0(x,y), & (x,y) \in \Omega, \\
u(0,y,t) = g(y,t), & y \in (0, B), \ t \in (0, T),\\
u(x,0,t)=u(x,B,t)=0 & x\in (0,\infty), \ t\in (0,T),
\end{cases}\end{equation}
where $u_{0}$ and $g$ satisfy appropriate regularity hypotheses and $T>0$. This problem will be solved in $\mathcal{S}'$ using an initial value problem given by:
\begin{equation}\label{PVI}\begin{cases}
\tilde{u}_t + a\tilde{u}_{x} + \tilde{u}_{xxx} + \tilde{u}_{xyy} + \tilde{u} \tilde{u}_x = \delta_{0}(x)f(y,t), & (x,y) \in \mathbb{R}^{2}, \ t \in (0, T); \\
\tilde{u}(x,y,0) = \tilde{u}_0(x,y), & (x,y) \in \mathbb{R}^{2};
\end{cases}\end{equation}
where $\delta_{0}$ is the Dirac delta at $x=0$ and $\tilde{u}_{0}$ is a suitable extension of $u_{0}$. The forcing function $f$ must be selected such that
\begin{eqnarray}
    \tilde{u}(0,y,t)=\tilde{g}(y,t), &\ t\in (0,T),\label{x=0}\\
    \tilde{u}(x,0,t)=\tilde{u}(x,B,t)=0, &\ t\in (0,T),\label{y=B=0}
\end{eqnarray}
where $\tilde{g}$ is a $L_{y}^{2}$ extension of $g$. Constructing the solution $\tilde{u}$ of the IVP \eqref{PVI}, we obtain the solution $u$ of the IBVP \eqref{PVIC} by restriction
\begin{equation}
    u=\tilde{u}|_{\lbrace x>0\rbrace\times(0,B)\times (0,T)}.
\end{equation}
The linear formulation of the IVP \eqref{PVI}, obtained by suppressing the advective term $\tilde{u}\tilde{u}_{x}$, is solved using Duhamel's formula,
\begin{equation}
    \tilde{u}(x,y,t)=S(t)\tilde{u}_{0}(x,y)+\int_{0}^{t}S(t-t')\delta_{0}(x)f(y,t')\ dt',
\end{equation}
where $S(t)$ is the linear operator given by
\begin{equation}\label{S(t)u0}
    S(t)v(x,y) = \int_{\mathbb{R}^2} e^{i(x\xi + y\eta + t(\xi^{3}+\xi\eta^{2}-a\xi))} \widehat{v}(\xi, \eta) \, d\xi d\eta
\end{equation}
with
\begin{equation}
    \hat{v}(\xi, \eta)=\mathcal{F}_{x,y}[v(x,y)](\xi,\eta)=\int_{\mathbb{R}^{2}} e^{-i(x\xi+y\eta)}v(x,y)\, dx\, dy.
\end{equation}
Considering $A(x,y)$ defined by \eqref{airyeq2d}, the semigroup \eqref{S(t)u0} becomes
\begin{equation}\label{S(t)airy}
S(t)v(x,y)=\frac{1}{t^{2/3}} \int_{\mathbb{R}^2} A\left( \frac{x-at-x'}{t^{1/3}}, \frac{y-y'}{t^{1/3}} \right) v(x',y') \, dx' dy'.
\end{equation}

\section{Boundary forcing}

\subsection{Inverse boundary operator}

We describe how to select the forcing term $f$ in the above context and reveal the natural regularity relations between $\tilde{u}_0$, $\tilde{g}$, and $f$. The boundary condition \eqref{x=0} determines the forcing function $f$ if one solves
\begin{equation} \label{eq:contorno_forca}
    \left. \int_0^t S(t-t')\delta_0(x) f(y,t') \, dt' \right|_{x=0} = \tilde{g}(y,t) - \left. S(t)\tilde{u}_0(x,y) \right|_{x=0}
\end{equation}
for $f$ in terms of $\tilde{g}$. The right-hand side of \eqref{eq:contorno_forca} suggests that the regularity properties of $\tilde{g}(y,t)$ should be at least the same as those expressed by the linear trace $\left. S(t)\tilde{u}_0(x,y) \right|_{x=0}$. For $\tilde{u}_0 \in L^2(\mathbb{R}^2)$, the local smoothing argument shows that $\left. S(t)\tilde{u}_0 \right|_{x=0} \in H^{1/3}(\mathbb{R}_t; L^2_y)$, as we prove below in Section \ref{5.2}.

Denote $\tilde{g}_1(y,t) = \tilde{g}(y,t) - \left. S(t)\tilde{u}_0(x,y) \right|_{x=0} \in H^{1/3}_{t}L^{2}_{y}$. Then \eqref{S(t)airy} and \eqref{eq:contorno_forca} imply
\begin{equation}\label{g1tilde}
    \tilde{g}_1(y,t)=\int_0^t \int_{\mathbb{R}} \frac{1}{(t-t')^{2/3}} A\left( -\frac{a(t-t')}{(t-t')^{1/3}}, \frac{y-y'}{(t-t')^{1/3}} \right) f(y', t') \, dy' dt'.
\end{equation}
Using \eqref{airyeq2d} backward, \eqref{g1tilde} gives
\begin{equation}
\tilde{g}_1(y,t)=\int_{0}^{t} (t-t')^{-2/3} \int_{\mathbb{R}^3} e^{-ia(t-t')^{2/3}\xi} e^{i \frac{(y-y')}{(t-t')^{1/3}} \eta} e^{i \xi^3} e^{i \xi\eta^2} \, d\xi \, d\eta \, f(y', t') \, dy' \, dt'.\end{equation}
Using the Fourier transform in $y'$, we factor out the $\xi$ terms to obtain
\begin{equation}\label{factorxi}
\tilde{g}_1(y,t)=\int_{0}^{t} (t-t')^{-2/3} \int_{\mathbb{R}} e^{i \frac{y}{(t-t')^{1/3}} \eta}\left(\int_{\mathbb{R}} e^{-i \xi \left(a(t-t')^{2/3} - \eta^2\right)} e^{i \xi^3} d\xi\right)\, \hat{f} \left( \frac{\eta}{(t-t')^{1/3}}, t' \right) d\eta \, dt'.\end{equation}
Making the change of variable $w = \xi^3$, \eqref{factorxi} becomes
\begin{equation}\label{changew}\begin{aligned}
\tilde{g}_1(y,t)=&\int_{0}^{t} (t-t')^{-2/3} \int_{\mathbb{R}} e^{i \frac{y}{(t-t')^{1/3}} \eta} \left(\int_{\mathbb{R}}e^{-i \left(a(t-t')^{2/3} - \eta^2\right) w^{1/3}} e^{iw(1)}\frac{dw}{3w^{2/3}}\right)\\ &\hat{f}\left(\frac{\eta}{(t-t')^{1/3}}, t'\right) \, d\eta \, dt'.\end{aligned}\end{equation}
We use the fact that
\begin{equation}
    \mathcal{F}^{-1} \left( \frac{e^{-i (a(t-t')^{2/3} - \eta^2) w^{1/3}}}{w^{2/3}} \right) (1) = C \left( 1 - (a(t-t')^{2/3} - \eta^2) \right) \text{sgn}\left( 1 - (a(t-t')^{2/3} - \eta^2) \right),
\end{equation}
where $C$ is the constant generated by the inverse Fourier transform. Note that here we choose $t>0$ to be sufficiently small. Thus, \eqref{changew} becomes
\begin{equation}\label{eqresfor} C \int_{0}^{t}\left[ (t-t')^{-1/3} f(y,t') -a(t-t')^{1/3} f(y,t') - (t-t')^{1/3} \partial_y^2 f(y,t')\right] \, dt'=\tilde{g}_{1}(y,t).\end{equation}
Recalling the theory of the Riemann-Liouville fractional integral \cite{CK}, given by
\begin{equation}
\mathcal{I}_{\gamma}[h](t) = \frac{1}{\Gamma(\gamma)} \int_0^t (t-s)^{\gamma-1} h(s) \, ds, \quad t > 0,\end{equation}
we write equation \eqref{eqresfor} to be
\begin{equation}
C\left( \Gamma(2/3) \mathcal{I}_{2/3} f(y,t) - \Gamma(4/3) a \mathcal{I}_{4/3} f(y,t) - \Gamma(4/3) \mathcal{I}_{4/3} \partial_y^2 f(y,t)\right)=\tilde{g}_{1}(y,t).
\end{equation}
Apply $\mathcal{I}_{-2/3}$ to the equation, using the facts that $\mathcal{I}_{\gamma}\mathcal{I}_{\beta}[h]=\mathcal{I}_{\gamma+\beta}[h]$ and $\mathcal{I}_{0}=I\!d$. The result gives
\begin{equation}
    f(y,t) - \frac{\Gamma(4/3)}{\Gamma(2/3)} a \mathcal{I}_{2/3} f(y,t) - \frac{\Gamma(4/3)}{\Gamma(2/3)} \mathcal{I}_{2/3} \partial_y^2 f(y,t)=\frac{\mathcal{I}_{-2/3}\tilde{g}_{1}(y,t)}{C\ \Gamma(2/3)}.
\end{equation}
Applying the Fourier transform in $y$ and rearranging, this yields
\begin{equation}
    \left[I\!d - C_{\Gamma}(a+\partial^{2}_{y})\mathcal{I}_{2/3} \right] \mathcal{F}_{y}f(\eta,t)=\frac{\mathcal{I}_{-2/3}\mathcal{F}_{y}\tilde{g}_{1}(\eta,t)}{C\ \Gamma(2/3)},
\end{equation}
where $C_{\Gamma}=\Gamma(4/3)/\Gamma(2/3)<1$. Thus, the operator on the left-hand side is invertible via the Neumann series \cite{podlubny1999}, and can be written as
\begin{equation}\label{f(y,t)RL}
    f(y,t)=\sum_{n=0}^{\infty} C_{\Gamma}^{n}(\mathcal{I}_{2/3}(a+\partial_{y}^{2}))^{n}\left(\frac{\mathcal{I}_{-2/3}\tilde{g}_{1}(y,t)}{C\ \Gamma(2/3)} \right).
\end{equation}

\subsection{Spectral reduction}

We expand $\tilde{g}_1(y,t)$ in an orthogonal sine series in the variable $y \in (0,B)$:
\begin{equation}
\tilde{g}_1(y,t) = \sum_{k=1}^{\infty} a_k(t) \sin\left(\frac{k\pi y}{B}\right)\end{equation}
with
\begin{equation}
a_k(t) = \frac{2}{B} \int_0^B \tilde{g}_1(y,t) \sin\left(\frac{k\pi y}{B}\right) \, dy.\end{equation}
Furthermore, the condition $\tilde{g}_1 \in H^{1/3}_t L^2_{y}$ ensures that the sequence of coefficients satisfies
\begin{equation}
\sum_{k=1}^{\infty} k^2 \|a_k\|_{H^{1/3}_{t}}^2 \sim \|\tilde{g}_1\|_{H^{1/3}_t L^2_y}^2 < \infty.
\end{equation}
Note that the differential operator $a+\partial^{2}_{y}$ acts purely as a scalar multiplier. Indeed, denote the associated eigenvalues by $\sigma_{k} = a - \left(\frac{k\pi}{B}\right)^2$. Projecting the $n$-th term of the series $f(y,t)$ onto the $k$-th spatial mode and into the temporal Fourier space, we obtain:
\begin{equation}\widehat{f}_{n,k}(\tau) = C_{\Gamma}^{n} \sigma_{k}^{n} (i\tau)^{-2n/3} \left[ \frac{(i\tau)^{2/3}}{C \Gamma(2/3)} \widehat{a}_{k}(\tau) \right].\end{equation}
Summing over $n \geq 0$, the complete transform of the $k$-th mode of $f$ can be written as a geometric series with the ratio $r_{k}(\tau) = C_{\Gamma} \sigma_{k} (i\tau)^{-2/3}$ as follows:
\begin{equation}
\widehat{f}_{k}(\tau) = \sum_{n=0}^{\infty} \left[ C_{\Gamma} \sigma_{k} (i\tau)^{-2/3} \right]^{n} \left[ \frac{(i\tau)^{2/3}}{C \Gamma(2/3)} \widehat{a}_{k}(\tau) \right].\end{equation}
Evaluating the sum of this series, $\frac{1}{1-r_k(\tau)}$, and simplifying the fractional terms, we isolate the multiplier of the inverse operator to get
\begin{equation}
\widehat{f}_{k}(\tau) = \frac{(i\tau)^{4/3}}{(i\tau)^{2/3} - C_{\Gamma} \sigma_{k}} \cdot \frac{1}{C \Gamma(2/3)} \widehat{a}_{k}(\tau).\end{equation}
To determine the regularity of $f(y,t)$, it suffices to analyze the asymptotic behavior of the resulting multiplier $M(k,\tau) = \frac{(i\tau)^{4/3}}{(i\tau)^{2/3} - C_{\Gamma} \sigma_{k}}$. Since $\sigma_{k} \sim -k^{2}$ for sufficiently high spatial frequencies, the denominator is strictly bounded below by a combination of both frequencies, ensuring that the multiplier satisfies
\begin{equation}
|M(k,\tau)| \sim \frac{|\tau|^{4/3}}{|\tau|^{2/3} + k^2} \leq  c|\tau|^{2/3}.
\end{equation}
Consequently,
\begin{equation}\label{serief}
\|f\|_{H^{-1/3}_t L^2_y}^2 \sim \sum_{k=1}^{\infty} k^2 \int_{\mathbb{R}} \langle \tau \rangle^{-2/3} |M(k,\tau)|^2 |\widehat{a}_k(\tau)|^2 \, d\tau \leq  c\sum_{k=1}^{\infty} k^2 \int_{\mathbb{R}} \langle \tau \rangle^{2/3} |\widehat{a}_k(\tau)|^2 \, d\tau < \infty.\end{equation}
Thus, by using the sine basis in $y$, the series in \eqref{serief} converges, and the forcing $f$ in terms of $g$ is given by equations \eqref{g1tilde} and \eqref{f(y,t)RL} with $f \in H^{-1/3}(\mathbb{R}_{t} ; L^{2}_{y}(0,B))$ where $H^{-1/3}=\left(H^{1/3}\right)'$ in the sense of Proposition 2.1 from \cite{CK}.

\section{Linear problem and modified Bourgain spaces}

\subsection{Linear setting}

To prove the existence and uniqueness of the solution to \eqref{PVIC}, we compose the Fourier series expansion with the theory of Bourgain spaces adapted for domains with boundaries. Let us consider the following linear IBVP:
\begin{equation}\label{PVIC_linear}
\begin{cases}
u_t + au_x + u_{xxx} + u_{xyy} = 0, & (x, y) \in \Omega, \,\, t > 0, \\
u(x, y, 0) = u_{0}(x, y), & u_{0} \in L^{2}(\Omega), \\
u(0, y, t) = g(y, t), & g \in H^{1/3}_{t} L^{2}_{y}, \\
u(x, 0, t) = u(x, B, t) = 0, & x \in (0,\infty), \,\, t > 0.
\end{cases}\end{equation}
Given the homogeneous Dirichlet conditions at $y=0$ and $y=B$, the differential operator in $y$ is self-adjoint and allows to expand the solution in orthogonal sine series.
We project $u(x,y,t)$, $u_0(x,y)$ and $g(y,t)$ onto the basis $e_{k}(y) = \sqrt{\frac{2}{B}} \sin\left(\frac{k\pi y}{B}\right)$ to be $$u(x, y, t) = \sum_{k=1}^{\infty} u_{k}(x, t) e_{k}(y),$$$$u_0(x, y) = \sum_{k=1}^{\infty} u_{k,0}(x) e_{k}(y),$$$$g(y, t) = \sum_{k=1}^{\infty} g_{k}(t) e_{k}(y).$$

For each $k \in \mathbb{N}$, the mode $u_{k}(x,t)$ solves a linear KdV-type equation on the half-line $x > 0$:\begin{equation}\partial_t u_{k} + (a-\lambda_{k}) \partial_x u_{k} + \partial_x^3 u_{k} = 0,\end{equation}
with $\lambda_k = \left(\frac{k\pi}{B}\right)^2$.

The conditions for each $k$ become $u_{k}(x, 0) = u_{k,0}(x) \in L^2(\mathbb{R}^+)$ and $u_{k}(0, t) = g_{k}(t) \in H^{1/3}_t$.

Considering the dispersive characteristics of the operator and aiming to avoid the degeneracy in the low-frequency region, we use the modified version of the Bourgain spaces: for $0<b<1/2$ and a fixed parameter $1/2 < \alpha < 2/3$, we define the space $X^{0,b}$ by the norm:
\begin{equation}
\|u\|_{X^{0,b}}^2 := \sum_{k=1}^{\infty} \left( \underset{|\omega_k(\xi)| \ge 1}{\iint} \langle \tau-\omega_k(\xi)\rangle^{2b} |\widehat{u}_k(\xi,\tau)|^2 d\xi d\tau + \underset{|\omega_k(\xi)| < 1}{\iint} \langle \tau \rangle^{2\alpha} |\widehat{u}_k(\xi,\tau)|^2 d\xi d\tau \right),\end{equation}
where $\langle\ .\ \rangle=1+|\ .\ |$, $\omega_{k}(\xi)=\xi^{3}+(\lambda_{k}-a)\xi$ and $\widehat{u}_k(\xi,\tau)$ is the Fourier transform in $x$ and $t$ of $u_k$. The corresponding space for the non-homogeneous (or forcing) term, denoted by $Y^{0,b}$, is defined analogously by the norm:
\begin{eqnarray}
    \|f\|_{Y^{0,b}}^2 &:=& \sum_{k=1}^{\infty} \left( \underset{|\omega_k(\xi)| \ge 1}{\iint} \frac{|\widehat{f}_k(\xi,\tau)|^2}{\langle \tau-\omega_k(\xi)\rangle^{2b}} d\xi d\tau + \underset{|\omega_k(\xi)| < 1}{\iint} \frac{|\widehat{f}_k(\xi,\tau)|^2}{\langle \tau \rangle^{2(1-\alpha)}} d\xi d\tau \right)\\ &+& \sum_{k=1}^{\infty} \underset{|\omega_{k}(\xi)| \ge 1}{\int} \left( \underset{|\omega_{k}(\xi)| \ge 1}{\int} \frac{|\widehat{f}_{k}(\xi,\tau)|}{\langle \tau-\omega_{k}(\xi)\rangle} d\tau \right)^{2} d\xi.\nonumber
\end{eqnarray}
For the study of the initial boundary value problem in a finite time $T > 0$, we consider the standard restriction of the space-time spaces, \cite{MP}. We define $X^{0,b}_T$ as the space of restrictions to the domain $t \in [0,T]$, equipped with the usual quotient norm:
\begin{equation}\|u\|_{X^{0,b}_T} := \inf \big\{ \|v\|_{X^{0,b}} \,:\, v \in X^{0,b} \text{ with } v|_{[0,T]} = u \big\}.\end{equation}

Finally, to ensure the necessary regularity in the nonlinear estimates and in the control of the spatial traces on the boundary, the full solution space $Z^b_T$ for the local theory is structured as the intersection of the modified Bourgain space with the natural energy space of the equation. For $0< b < 1/2$ and $1/2<\alpha <2/3$, we define:
\begin{equation}
Z^{b}_{T} := X^{0,b}_{T} \cap L^{2}(0,T; \dot{H}^{1}_{x,y}),\end{equation}
equipped with the sum norm $\|u\|_{Z^{b}_{T}} = \|u\|_{X^{0,b}_{T}} + \|u\|_{L^{2}_{T} \dot{H}^{1}_{x,y}}$.

\subsection{Linear estimates}\label{5.2}

\begin{lemma}
Let $\theta \in C_{0}^{\infty}(\mathbb{R})$ be a cutoff function supported in $[-T, T]$, where $T<1$. For any $u_{0}\in L^{2}(\Omega)$, we have
\begin{equation}
\|\theta(t)S(t)u_0\|_{X_{T}^{0,b}} \leq C_\theta \|u_0\|_{L^2_{x,y}}.
\end{equation}
\end{lemma}
\begin{proof}
Projecting the initial data onto the transverse basis $\{e_k\}_{k \ge 1}$, the linear evolution for each $k$ is given by $u_k(x,t) = \theta(t)S_k(t)u_{0,k}(x)$. Thus, we have that $\widehat{u}_k(\xi, \tau) = \widehat{\theta}(\tau - \omega_k(\xi)) \widehat{u}_{0,k}(\xi)$.

Evaluating the norm in $X_{T}^{0,b}$ and performing the change of variables $\sigma = \tau - \omega_k(\xi)$, the contribution of the high-frequency regime ($|\omega_{k}(\xi)| > 1$) is bounded as
\begin{equation}
\sum_{k=1}^{\infty} \underset{|\omega_{k}(\xi)| > 1}{\iint} \langle \tau - \omega_k(\xi) \rangle^{2b} |\widehat{\theta}(\sigma)|^2 |\widehat{u}_{0,k}(\xi)|^2 d\sigma d\xi \le \|\theta\|_{H^b_t}^2 \sum_{k=1}^{\infty} \|u_{0,k}\|_{L^2_x}^2.
\end{equation}

For the low-frequency region ($|\omega_{k}(\xi)| \le 1$), we use the triangle inequality given by $\langle \tau \rangle^{2b} \leq C\left(\langle \tau - \omega_k(\xi) \rangle^{2b} + \langle \omega_k(\xi) \rangle^{2b}\right)$. Since the dispersion relation $\omega_k(\xi)$ is uniformly bounded in the region $|\omega_{k}(\xi)| \le 1$, the spatial weight is absorbed (by the regularity of $\theta$) to hold
\begin{equation}
\sum_{k=1}^{\infty} \underset{|\xi| \le 1}{\iint} \langle \tau \rangle^{2b} |\widehat{u}_k(\xi,\tau)|^2 d\tau d\xi \le C(\theta) \sum_{k=1}^{\infty} \|u_{0,k}\|_{L^2_x}^2.
\end{equation}

The proof is concluded by adding both estimates and applying Parseval's identity in $y$, which guarantees that $\sum_{k=1}^{\infty} \|u_{0,k}\|_{L^2_x}^2 = \|u_0\|_{L^2_{x,y}}^2$.
\end{proof}

\begin{lemma}\label{lema5.3}
Let $f\in H^{-1/3}_{t}L^{2}_{y}$ and let $\theta$ be a $C_{0}^{\infty}(\mathbb{R})$ cutoff function. We have \begin{equation}\label{lemma5.3}
\left\|\theta(t)\int_{0}^{t} S(t-t')\delta_{0}(x)f(y,t')dt'\right\|_{X^{0,b}}\leq C\|f\|_{H^{-1/3}_{t}L^{2}_{y}}.\end{equation}
\end{lemma}
\begin{proof}
Denote by $f_{k}$ the projection of a function $f$ onto the transverse basis $\{e_{k}\}_{k \ge 1}$, where $k$ represents the transverse index. We observe that $\delta_{0}(x)$ and $\theta(t)$ are independent of the transverse variable, thus commuting with the projection operator. Furthermore, for each fixed $k$, $S(t)$ acts on the $x$-variable in a manner equivalent to the one-dimensional case, denoted by $S_{k}(t)$, associated with the symbol of the parameterized dispersion relation. Projecting the expression from the left-hand side of \eqref{lemma5.3} onto the transverse basis, we obtain, for each $k$
\begin{equation}
u_{k}(x,t) = \theta(t)\int_{0}^{t} S_{k}(t-t')\delta_{0}(x)f_{k}(t')dt'.\end{equation}
For any fixed $k$, the function $t' \mapsto f_{k}(t')$ lies in the one-dimensional space $H^{-1/3}_{t}$. Thus, the representation reduces the multivariable operator to a family of one-dimensional integral equations in the variables $x$ and $t$. Due to Lemma 5.3 from \cite{CK}, there exists a constant $C$ (independent of $k$) such that:
\begin{equation}\label{1Dlema5.3}
\|u_{k}\|_{X^{0,b}_{x,t}}^{2} \leq C^{2} \|f_{k}\|_{H^{-1/3}_{t}}^{2}.\end{equation}
By definition, and by Parseval's identity with respect to the transverse basis, it holds
\begin{equation}
\|u\|_{X^{0,b}_{x,y,t}}^{2} = \sum_{k=1}^{\infty} \|u_{k}\|_{X^{0,b}_{x,t}}^{2},\end{equation}
which gives
\begin{equation}\label{1Dlema5.3a}
\|u\|_{X^{0,b}_{x,y,t}}^{2} \leq C^{2} \sum_{k=1}^{\infty} \|f_{k}\|_{H^{-1/3}_{t}}^{2}.\end{equation}
Applying Parseval's identity again to the right-hand side of \eqref{1Dlema5.3a}, the result reads
\begin{equation}\sum_{k=1}^{\infty} \|f_{k}\|_{H^{-1/3}_{t}}^{2} = \|f\|_{H^{-1/3}_{t}L^{2}_{y}}^{2}.\end{equation}
Therefore, $\|u\|_{X^{0,b}_{x,y,t}}^{2} \leq C^{2} \|f\|_{H^{-1/3}_{t}L^{2}_{y}}^{2}$, which yields \eqref{lemma5.3}.
\end{proof}

We now put our attention to the forced Duhamel term $\int_0^t S(t-t')w(x,y,t')dt'$. The next two lemmas describe its mapping properties. We first show that the inclusion of a smooth temporal cutoff $\theta$ yields a bounded map from $Y^{0,b}$ to $X^{0,b}$. Then, given $w \in Y^{0,b}$, we evaluate the boundary trace of this localized response, showing it to be in $H^{1/3}_t L^2_y$.

\begin{lemma}
For $w\in Y^{0,b}$ and $\theta \in C_{0}^{\infty}$, we have
\begin{equation}
\left\|\theta(t)\int_{0}^{t} S(t-t')w(x,y,t')dt'\right\|_{X^{0,b}}\leq C\|w\|_{Y^{0,b}}.
\end{equation}
\end{lemma}
\begin{proof}
Recall that the function $\theta(t)$ is independent of $x$ and the transverse variable, and $S(t)$ acts as the multiplier $e^{-it(\xi^{3}+(\lambda_{k}-a)\xi)}$; the projection onto the transverse basis commutes with the Duhamel operator. Thus, projecting the function
\begin{equation}u(x,y,t) = \theta(t)\int_{0}^{t} S(t-t')w(x,y,t')dt'\end{equation} onto the transverse basis, we obtain a decoupled family of equations for each $k$:
\begin{equation}u_{k}(x,t) = \theta(t)\int_{0}^{t} S_{k}(t-t')w_{k}(x,t')dt',\end{equation} where $S_{k}(t)$ is constructed as in the proof of Lemma \ref{lema5.3}. For a fixed $k$, the above identity corresponds to Lemma 5.4 from \cite{CK} for the one-dimensional case; that is, there exists a constant $C > 0$, independent of $k$, such that:
\begin{equation}\label{1Dlema5.4}
\|u_{k}\|_{X^{0,b}_{x,t}}^{2} \leq C^{2} \|w_{k}\|_{Y^{0,b}_{x,t}}^{2},\end{equation}
where the norms $X^{0,b}_{x,t}$ and $Y^{0,b}_{x,t}$ represent the one-dimensional analogues restricted to the transverse frequency $k$. Summing \eqref{1Dlema5.4} over $k$, we obtain:
\begin{equation}\sum_{k=1}^{\infty} \|u_{k}\|_{X^{0,b}_{x,t}}^{2} \leq C^{2} \sum_{k=1}^{\infty} \|w_{k}\|_{Y^{0,b}_{x,t}}^{2}.\end{equation}
By the definition of the Bourgain space norms $X^{0,b}$ and $Y^{0,b}$, and by Parseval's identity with respect to the transverse basis, this means that
\begin{equation}
\|u\|_{X^{0,b}}^{2} \leq C^{2} \|w\|_{Y^{0,b}}^{2},\end{equation} and the result follows.
\end{proof}

\begin{lemma}
For $h\in Y^{0,b}$ and $\theta \in C_{0}^{\infty}$, it holds
\begin{equation}
\left\|\theta(t)\int_{0}^{t} S(t-t')h(x,y,t')dt'\Big|_{x=0}\right\|_{H^{1/3}_{t}L^{2}_{y}}\leq C\|h\|_{Y^{0,b}}.
\end{equation}
\end{lemma}
\begin{proof}
Let $v(y,t)$ be the restriction of the Duhamel term at $x=0$. Since $\theta(t)$ and the trace operator commute with the projection onto the transverse basis, we decompose the trace to be
\begin{equation}
v_{k}(t) = \theta(t)\int_{0}^{t} S_{k}(t-t')h_{k}(x,t')dt'\Big|_{x=0},\end{equation}
where $S_{k}(t)$ represents the one-dimensional linear operator with the dispersion relation $\omega_{k}(\xi)$.

At this point, we apply Lemma 5.5 from \cite{CK} for the boundary estimate of the one-dimensional non-homogeneous term; that is, there exists a constant $C > 0$, independent of $k$, such that
\begin{equation}\label{1Dlema5.5}
\|v_{k}\|_{H^{1/3}_{t}}^{2} \leq C^{2} \|h_{k}\|_{Y^{0,b}_{x,t}}^{2}.\end{equation}
Finally, summing \eqref{1Dlema5.5} over $k$ and Parseval's identity with respect to the transverse basis, imply
\begin{equation}
\|v\|_{H^{1/3}_{t}L^{2}_{y}}^{2} \leq C^{2} \|h\|_{Y^{0,b}}^{2},
\end{equation}
which completes the proof.
\end{proof}

\section{Nonlinear estimate}

In this subsection, we prove an estimate required to handle the
nonlinearity of the ZK equation. This is the main step to set the
fixed-point argument and conclude the well-posedness result.

\begin{lemma}
For $v \in Z^{b} = X^{0,b} \cap L^{2}_{t} H^{1}_{x,y}$, we have, for suitable $3/8<b<1/2$ and $1/2<\alpha<2/3$, the bilinear estimate
\begin{equation}\label{bilinear}
\|\partial_{x}(v^{2})\|_{Y^{0,b}}\leq C\|v\|_{Z^{b}}^{2}.
\end{equation}
\end{lemma}
\begin{proof}
Let $f = \partial_{x}(v^2)$. Thus, we have $\widehat{f}_{k}(\xi,\tau) = i\xi \mathcal{F}(v^{2})_{k}(\xi,\tau)$. To prove \eqref{bilinear}, we must estimate the three components of the $Y^{0,b}$ norm. To this end, we split the domain of integration into two regions: $|\omega_{k}(\xi)| < 1$ and $|\omega_{k}(\xi)| \ge 1$.

\textbf{Case A.} $|\omega_{k}(\xi)| < 1$.

In this region, the structure of $\omega_{k}(\xi)$ ensures that the
domain is bounded in the spatial frequency variable; that is, there
exists a constant $M_0 > 0$ such that $|\xi| \le M_0$. Consider the
second term of the $Y^{0,b}$ norm given by
\begin{equation}I_{low} = \sum_{k=1}^{\infty} \iint_{|\omega_{k}| < 1} \frac{|\xi|^{2} |\mathcal{F}(v^{2})_{k}(\xi,\tau)|^{2}}{\langle \tau \rangle^{2(1-\alpha)}} d\xi d\tau.\end{equation}
Since $\langle \tau \rangle^{-2(1-\alpha)} \le 1$ for all $\tau \in
\mathbb{R}$, and using the bound on the spatial frequency, we
estimate $I_{low}$ by
\begin{equation}
    I_{low} \le M_0^2 \sum_{k=1}^{\infty} \iint_{\mathbb{R}^2} |\mathcal{F}(v^2)_k(\xi,\tau)|^2 d\xi d\tau.
\end{equation}
By Plancherel's theorem, we obtain that
\begin{equation}
    I_{low} \le M_{0}^{2} \|v^{2}\|_{L^{2}_{x,y,t}}^{2} = M_{0}^{2} \|v\|_{L^{4}_{x,y,t}}^{4}.
\end{equation}
By the Strichartz estimate for the ZK equation, the space $X^{0,b}$
embeds continuously into $L^{4}_{x,y,t}$ provided that $b > 3/8$.
Since $b \in (3/8, 1/2)$, there exists a constant $C_{S} > 0$ such
that
\begin{equation}
    I_{low} \le M_{0}^{2} C_{S}^{4} \|v\|_{X^{0,b}}^{4} \le C_{1} \|v\|_{Z^{b}}^{4}.
\end{equation}

\textbf{Case B.} $|\omega_{k}(\xi)| \ge 1$.

The contribution of this region to the $Y^{0,b}$ norm decomposes
into two terms, $I_{high,1}$ and $I_{high,2}$, to be analyzed
separately via duality arguments.

\textbf{Case B1.} Estimate for $I_{high,1}$.

Let the corresponding term of the $Y^{0,b}$ norm be given by
\begin{equation}
    I_{high,1} = \sum_{k=1}^{\infty} \iint_{|\omega_{k}(\xi)| \ge 1} \frac{|\xi|^{2}
    |\mathcal{F}(v^{2})_{k}(\xi,\tau)|^{2}}{\langle \tau - \omega_{k}(\xi)\rangle^{2b}} d\xi d\tau.
\end{equation}
By duality in $L^2$, we introduce a non-negative test function $g_{k}(\xi,\tau)$ with $\|g\|_{L^2} \le 1$.
We must evaluate the linear functional
\begin{equation}\label{6.6}
    J_{1} = \sum_{k=1}^{\infty} \iint_{|\omega_{k}(\xi)|\geq 1} \frac{|\xi| |\mathcal{F}(v^{2})_k(\xi,\tau)|}{\langle \tau
    - \omega_{k}(\xi)\rangle^{b}} g_{k}(\xi,\tau) d\xi d\tau.
\end{equation}
Defining the positive function $\widehat{V}_{j} = |\widehat{v}_{j}|$
and expanding the product into a convolution, we bound $J_{1}$ by
its extended form over $\mathbb{R}^{3}$. By the symmetry of the
convolution, we integrate over the region where $|\xi_{1}| \ge
|\xi_{2}|$, which implies $|\xi| \le |\xi_{1}| + |\xi_{2}| \le
2|\xi_{1}|$. Hence, \eqref{6.6} becomes
\begin{equation}
    J_{1} \leq C\sum_{k,k_{1}} \iiint_{\mathbb{R}^{3}} \left[ \frac{g_{k}(\xi,\tau)}{\langle \tau -
    \omega_{k}(\xi)\rangle^{b}} \right] \cdot \Big( |\xi_{1}| \widehat{V}_{k_{1}}(\xi_{1},\tau_{1}) \Big) \cdot \widehat{V}_{k_{2}}(\xi_{2},\tau_{2}) d\Gamma,
\end{equation}
where $d\Gamma$ denotes the measure under the constraint $\xi =
\xi_{1}+\xi_{2}$ and $\tau = \tau_{1}+\tau_{2}$. Applying
Plancherel's theorem and the generalized Hölder inequality with
indices $1/4 + 1/2 + 1/4 = 1$, we have
\begin{equation}
J_{1} \leq C\left\|\mathcal{F}^{-1}\left(\frac{g_{k}}{\langle \tau - \omega_{k}(\xi)
\rangle^{b}}\right)\right\|_{L^{4}} \cdot \left\|\mathcal{F}^{-1}\left(|\xi_1|
\widehat{V}_{k_{1}}\right)\right\|_{L^{2}} \cdot \left\|\mathcal{F}^{-1}\left(\widehat{V}_{k_{2}}\right)\right\|_{L^{4}}.
\end{equation}
Hence,
\begin{equation}
\left\|\mathcal{F}^{-1}\left(|\xi_1| \widehat{V}_{k_{1}}\right)\right\|_{L^{2}}= \|\partial_{x} v\|_{L^{2}} \leq \|v\|_{L^{2}_{t} H^{1}_{x,y}} \leq \|v\|_{Z^{b}},
\end{equation}
and
\begin{equation}
\left\|\mathcal{F}^{-1}\left(\widehat{V}_{k_{2}}\right)\right\|_{L^{4}} \leq C_{S} \|v\|_{X^{0,b}} \leq C_{S} \|v\|_{Z^{b}}.
\end{equation}
Since $\|g\|_{L^{2}} \le 1$, we have that $\mathcal{F}^{-1}\left(\frac{g_{k}}{\langle \tau - \omega_{k}(\xi) \rangle^{b}}\right) \in X^{0,b}$. Since $b > 3/8$, we apply the Strichartz embedding to obtain
\begin{equation}
\left\|\mathcal{F}^{-1}\left(\frac{g_{k}}{\langle \tau - \omega_{k}(\xi) \rangle^{b}}\right)\right\|_{L^4} \le C_{S}\left\|\mathcal{F}^{-1}\left(\frac{g_{k}}{\langle \tau - \omega_{k}(\xi) \rangle^{b}}\right)\right\|_{X^{0,b}} \le 1.\end{equation}
Therefore, $J_{1} \le C_{h,1}\|v\|_{Z^{b}}^{2}$, which implies $I_{high,1} \le C_{2}\|v\|_{Z^{b}}^{4}$.

\textbf{Case B2.} Estimate for $I_{high,2}$.

Let the corresponding term of the $Y^{0,b}$ norm be given by
\begin{equation}
I_{high,2} = \sum_{k=1}^{\infty} \int_{|\omega_{k}| \ge 1} \left( \int_{|\omega_{k}| \ge 1} \frac{|\xi| |\mathcal{F}(v^{2})_{k}(\xi,\tau)|}{\langle \tau - \omega_{k}(\xi) \rangle} d\tau \right)^{2} d\xi.
\end{equation}
By duality in the mixed norm $L^{2}_{\xi, k}$, we take test
functions $h_{k}(\xi)$ depending only on the spatial frequency, such
that $\|h\|_{L^{2}} \le 1$. Arguing via symmetry breaking, where
$|\xi| \le 2|\xi_{1}|$, and using Hölder's inequality as in case
B1, we get
\begin{equation}
I_{high,2}^{1/2} \le C\left\|\mathcal{F}^{-1} \left( \frac{h_{k}(\xi)}{\langle \tau - \omega_{k}(\xi) \rangle} \right)\right\|_{L^{4}_{x,y,t}} \cdot \left\|\mathcal{F}^{-1}\left(|\xi_1| \widehat{V}_{k_{1}}\right)\right\|_{L^{2}_{x,y,t}} \cdot \left\|\mathcal{F}^{-1}\left(\widehat{V}_{k_{2}}\right)\right\|_{L^{4}_{x,y,t}}.
\end{equation}
It remains to show that $\mathcal{F}^{-1} \left(
\frac{h_{k}(\xi)}{\langle \tau - \omega_{k}(\xi) \rangle} \right)
\in L^{4}_{x,y,t}$.

Fix $s$ such that $3/8 < s < 1/2$. The norm of $\mathcal{F}^{-1}
\left( \frac{h_{k}(\xi)}{\langle \tau - \omega_{k}(\xi) \rangle}
\right)$ in the space $X^{0,s}$ is given by
\begin{equation}
\left\|\mathcal{F}^{-1} \left( \frac{h_{k}(\xi)}{\langle \tau - \omega_{k}(\xi) \rangle} \right)\right\|_{X^{0,s}}^{2} = \sum_{k=1}^{\infty} \int_{\mathbb{R}} |h_{k}(\xi)|^{2} \left( \int_{\mathbb{R}} \frac{1}{\langle \tau - \omega_{k}(\xi) \rangle^{2(1 -s)}} d\tau \right) d\xi.
\end{equation}
Since $s < 1/2$, we have $2(1 - s) > 1$. Thus, the inner integral with respect to $\tau$ converges to a constant $C_{s} > 0$. It follows that
\begin{equation}
\left\|\mathcal{F}^{-1} \left( \frac{h_{k}(\xi)}{\langle \tau - \omega_{k}(\xi) \rangle} \right)\right\|_{X^{0,s}}^{2} \le C_{s} \sum_{k=1}^{\infty} \int_{\mathbb{R}} |h_{k}(\xi)|^2 d\xi = C_{s} \|h\|_{L^{2}_{\xi, k}}^{2} \le C_{s}.
\end{equation}
Since $\mathcal{F}^{-1} \left( \frac{h_{k}(\xi)}{\langle \tau - \omega_{k}(\xi) \rangle} \right) \in X^{0,s}$ with $s > 3/8$, the Strichartz embedding guarantees that
\begin{equation}\left\|\mathcal{F}^{-1} \left( \frac{h_{k}(\xi)}{\langle \tau - \omega_{k}(\xi) \rangle} \right)\right\|_{L^{4}} \le C_{s}.\end{equation}
Since the factors $\mathcal{F}^{-1}\left(|\xi_{1}| \widehat{V}_{k_{1}}\right)$ and $\mathcal{F}^{-1}\left(\widehat{V}_{k_{2}}\right)$ are bounded by $\|v\|_{Z^{b}}$ and $C_{S}\|v\|_{Z^{b}}$ respectively, we conclude that $I_{high,2}^{1/2} \le C_{h,2}\|v\|_{Z^{b}}^{2}$, that is, $I_{high,2} \le C_{3}\|v\|_{Z^{b}}^{4}$.

Combining the estimates for $I_{low}$, $I_{high,1}$, and $I_{high,2}$, the result follows.
\end{proof}

\section{Local well-posedness}

Recall the main result to be proven (see Section \ref{sub2.3}) which
establishes a local well-posedness for the IBVP \eqref{PVIC} with
the data $(u_{0}, g)\in L^{2}_{x,y}\times H^{1/3}_{t}L^{2}_{y}$.
\begin{proposition}{Theorem}
For any $u_{0}\in L^{2}(\Omega)$ and $g\in H^{1/3}_{t}(0,T;
L^{2}_{y}(0,B))$, there exist $T_{0}>0$, and a unique solution
$u(x,y,t)$ to \eqref{PVIC} such that
\begin{equation*}
u\in C([0,T_{0}];L^{2}(\Omega))\cap Z^{b}_{T_{0}}.
\end{equation*}
Furthermore, the data-to-solution map $(u_{0},g) \mapsto u$ is
Lipschitz continuous.
\end{proposition}
\begin{proof}
Define $E_{1}$ and $E_{2}$ as continuous linear extension operators (cf.\cite{CK}) such that
\begin{eqnarray}
\tilde{u}_{0} &=& E_{1}(u_{0}) \in L^{2}(\mathbb{R}^{2}) : \tilde{u}_{0}|_{\Omega} = u_{0} \text{ and } \|\tilde{u}_{0}\|_{L^{2}}
\le C_{1} \|u_{0}\|_{L^{2}},\\
\tilde{g} &=& E_{2}(g) \in H^{1/3}(\mathbb{R}_{t}; L^{2}(\mathbb{R}_{y})) : \tilde{g}|_{(0,B)\times(0,T)} = g \text{ and }
\|\tilde{g}\|_{H_{t}^{1/3}L_{y}^{2}} \le C_{2} \|g\|_{H_{t}^{1/3}L_{y}^{2}}.
\end{eqnarray}
Furthermore, let $\theta(t)$ be a smooth cutoff function such that $\theta(t) = 1$ for $t \in [-T, T]$ and $\text{supp}(\theta)
\subseteq [-2T, 2T]$. For $T > 0$, denote $\theta_{T}(t) = \theta(t/T)$.
Thus, for a solution $\tilde{u}$ to problem \eqref{PVI}, we obtain the global integral equation defined on $\mathbb{R}^2 \times \mathbb{R}$
\begin{equation}
\tilde{u} = \theta_{T}(t) S(t) \tilde{u}_{0} + \theta_{T}(t) \int_{0}^{t} S(t-t') \delta_{0}(x) f(y,t') dt' - \frac{1}{2} \theta_{T}(t) \int_{0}^{t} S(t-t') \partial_{x} (\tilde{u}^{2})(t') dt',
\end{equation}
where $f$ is defined by \eqref{f(y,t)RL}.

Defining the operator $\Phi(v)=\tilde{u}$, if we find a fixed point $v \in Z^{b}(\mathbb{R}^{2} \times \mathbb{R})$,
its restriction $u = v|_{\Omega_{T}}$ will be the desired local solution. Therefore, from the previous estimates, we have
\begin{eqnarray}
\|\theta_{T}(t) S(t) \tilde{u}_{0}\|_{X^{0,b}} &\le& C_{L} \|\tilde{u}_{0}\|_{L^2} \le C_{L} C_{1} \|u_{0}\|_{L^{2}(\Omega)}, \label{FP1}\\
\left\| \theta_{T}(t) \int_{0}^{t} S(t-t') \delta_{0}(x) f(y,t') dt' \right\|_{X^{0,b}} &\le& C_{B} \|\tilde{g}\|_{H_{t}^{1/3}L_{y}^{2}} \le C_{B} C_{2} \|g\|_{H_{t}^{1/3}L_{y}^{2}}, \label{FP2}\\
\left\| \theta_{T}(t) \int_{0}^{t} S(t-t') \partial_{x} (v^{2}) dt' \right\|_{X^{0,b}} &\le& C_{N} T^{\beta} \|\partial_{x} (v^{2})\|_{Y^{0,b}} \le \tilde{C}_{N} T^{\beta} \|v\|_{Z^{b}}^{2}. \label{FP3}
\end{eqnarray}
The exact value of $\beta > 0$ here can be computed like in either
\cite{LP}, or \cite{F}.

Write $M = C_{L} C_{1} \|u_{0}\| + C_{B} C_{2} \|g\|$ and set $K = \max\{1, \frac{1}{2}\tilde{C}_{N}\}$. For $v \in Z^{b}$, from \eqref{FP1}, \eqref{FP2} and \eqref{FP3}, we have
\begin{equation}
\|\Phi(v)\|_{Z^{b}} \le M + K T^{\beta} \|v\|_{Z^{b}}^{2}.
\end{equation}
Now, define: $B_{r} = \{ v \in Z^{b} \mid \|v\|_{Z^{b}} \le r \}$. For $\Phi(B_{r}) \subseteq B_{r}$, we require
\begin{equation}
\|\Phi(v)\|_{Z^{b}} \le M + K T^{\beta} r^{2} \le r.
\end{equation}
Choosing $r = 2M$, we have
\begin{equation}
\frac{r}{2} + K T^{\beta} r^{2} \le r \Rightarrow K T^{\beta} r \le \frac{1}{2}, \text{ i.e., } T^{\beta} \le \frac{1}{2Kr}.
\end{equation}
It remains to show the contraction. If $u, v \in Z^{b}$, we have
\begin{equation}
\|\Phi(u) - \Phi(v)\|_{Z^{b}} \le 2rK T^{\beta} \|u - v\|_{Z^{b}}.
\end{equation}
For a Lipschitz constant less than $1$ (say $1/2$), we require $2r K T^{\beta} \le 1/2$, to obtain
\begin{equation}
T^{\beta} \le \frac{1}{4Kr}.
\end{equation}

Therefore, there exists $T_0>0$ such that $B_r$ is a closed ball in
a Banach space $Z^{b}_{T_0}$, and $\Phi$ is a contraction. Hence,
there is a unique fixed point $v \in B_{r} \subset Z^{b}_{T_0}$.
Finally, since $u = v|_{\Omega_{T}}$, the result $u \in Z^{b}_{T_0}$
follows.
\end{proof}

\section*{Acknowledgments}
E. F. Avelino is supported by CAPES/Brazil - Finance Code 001.

\end{document}